\documentclass[10pt, reqno]{article} 
\usepackage{amsmath} 
\usepackage{amsthm} 
	
\theoremstyle{plain} 
		\newtheorem{lemma}{Lemma}
		\newtheorem{theorem}{Theorem}

\begin{document}

\title{Asymptotics of linearized cosmological perturbations}

\author{Paul T. Allen\\ 
Max-Planck-Institut f\"ur Gravitationsphysik\\
Albert-Einstein-Institut\\
Am M\"uhlenberg 1\\
14476 Potsdam, Germany\\
\textit{and}\\
Interdisciplinary Arts and Sciences\\
University of Washington, Tacoma\\
1900 Commerce Street\\
Tacoma, WA 98402, USA\\
\\
Alan D. Rendall \\ 
Max-Planck-Institut f\"ur Gravitationsphysik\\
Albert-Einstein-Institut\\
Am M\"uhlenberg 1\\
14476 Potsdam, Germany}

\date{}

\maketitle

\begin{abstract}
In cosmology an important role is played by homogeneous and isotropic 
solutions of the Einstein-Euler equations and linearized perturbations
of these. This paper proves results on the asymptotic behaviour of 
scalar perturbations both in the approach to the initial singularity of 
the background model and at late times. The main equation of interest is
a linear hyperbolic equation whose coefficients depend only on time.
Expansions for the solutions are obtained in both asymptotic regimes.
In both cases it is shown how general solutions with a linear equation of
state can be parametrized by certain functions which are coefficients in the 
asymptotic expansion. For some nonlinear equations of state it is found that 
the late-time asymptotic behaviour is qualitatively different from that in 
the linear case. 
\end{abstract}

\section{Introduction}

Astronomical observations allow information to be collected about the 
distribution of matter in the universe. This distribution contains
structures on many different scales. Astrophysicists would like to
provide a theoretical account of how these structures formed.
In particular, cosmologists would like to do this for structures on
the largest scales which can be observed. This means for instance giving
an explanation of the way in which galaxies cluster. The most powerful
influence on the dynamics of the matter distribution on very large scales 
is gravity. The most appropriate description of the gravitational field
in this context is given by the Einstein equations of general relativity.  
It is also necessary to choose a model of the matter which generates the
gravitational field. A frequent choice for this is a perfect fluid
satisfying the Euler equations. Thus, from a mathematical point of view,
the basic object of study is the Einstein-Euler system describing the 
evolution of a self-gravitating fluid. This is a system of quasilinear
hyperbolic equations.

The standard cosmological models are the Friedmann-Lemaitre-Robertson-Walker
(FLRW) models which are homogeneous and isotropic. This means in particular
that the unknowns in the Einstein-Euler system depend only on time and the 
partial differential equations reduce to ordinary differential equations.
With appropriate assumptions on the fluid these ODE's can be solved explicitly
or, at least, the qualitative behaviour of their solutions can be determined
in great detail. When it comes to the study of inhomogeneous structures, 
however, the FLRW models are by definition not sufficient. Since fully
inhomogeneous solutions of the Einstein-Euler system are difficult to 
understand a typical strategy is to linearize the system about a 
background FLRW model. Under favourable conditions the linearized 
perturbations could give information about the evolution under the 
Einstein-Euler system of initial data which are small but finite
perturbations of those for the FLRW background. 

Linearization about a highly symmetric solution is a classical practise in
applied mathematics. For some examples see \cite{turing}, \cite{chandrasekhar}
and \cite{kellersegel}. It should be noted, however, that there is an unusual
feature in the case of the Einstein-Euler system which has to do with the
fact that these equations are invariant under diffeomorphisms. This is 
related to the fact that the only thing that is of physical significance
are equivalence classes of solutions under diffeomorphisms. Since it is 
not known how to develop PDE theory in a manifestly diffeomorphism-invariant
way this leads to difficulties. There is a corresponding equivalence relation
on linearized solutions. Different linearized solutions are related by the
linearizations of one-parameter families of diffeomorphisms, which are known
in the literature on cosmology as gauge transformations. In the end what is
interesting is not the vector space of solutions of the linearized equations
but its quotient by gauge transformations. It is useful to represent this
quotient space by a subspace. This is what is known in the literature on
cosmology as gauge-invariant perturbation theory. This subject would no doubt
benefit from closer mathematical scrutiny but that task will not be attempted 
in the present paper.

Instead the following pragmatic approach will be adopted: take an equation
from the astrophysical literature on cosmological perturbation theory and 
analyse the properties of its solutions. As a basic source the book of 
Mukhanov \cite{mukhanov} will be used. The notation in the following will 
generally agree with that of \cite{mukhanov}. It is standard to classify 
cosmological perturbations into scalar, vector and tensor perturbations.
These terms will not be defined here. It should be noted that scalar
perturbations play a central role in the analysis of structure formation.
This motivates the fact that the results of this paper are concerned with
that case. After a suitable gauge choice scalar perturbations are 
described by solutions of a scalar wave equation for a function $\Phi$
which corresponds, roughly speaking, to the Newtonian gravitational 
potential. In order to get definite expressions for the Einstein-Euler
system and its linearization about an FLRW model it is necessary to 
choose an equation of state $p=f(\epsilon)$ for the fluid. Here $\epsilon$
is the energy density of the fluid and $p$ its pressure. A case which
is particularly simple analytically is that of a linear equation of state
$p=w\epsilon$ where $w$ is a constant. For physical reasons $w$ is
chosen to belong to the interval $[0,1]$. In fact the condition $w\ge 0$ 
is necessary in order to make the Euler equations hyperbolic. The case
$w=0$, known as dust, is somewhat exceptional and does not always fit well 
with the general arguments in the sequel. Since, however, dust frequently 
comes up in the literature on cosmology it is important to include it. In
those cases where the general argument fails for dust this will be pointed
out. 

For a linear equation of state as just described the equation for $\Phi$ is
\begin{equation}\label{basic}
\Phi''+\frac{6(1+w)}{1+3w}\frac1{\eta}\Phi'=w\Delta\Phi
\end{equation} 
Here a prime stands for $\frac{d}{d\eta}$. The time coordinate $\eta$ belongs
to the interval $(0,\infty)$. The spatial variables, which will be denoted 
collectively by $x$, are supposed to belong to the torus $T^3$. Thus 
periodic boundary conditions are imposed. The Laplacian is that of a fixed
flat metric on the torus. Its expression in adapted coordinates agrees with
that for the usual Laplacian on ${\bf R}^3$. As a consequence of standard
theory for linear hyperbolic equations this equation has a unique solution
on the whole time interval $(0,\infty)$ for appropriate initial data given at 
a fixed time $\eta=\eta_0>0$. These are the restrictions of $\Phi$ and 
$\Phi'$ to $\eta=\eta_0$.

In the following, after some background and notation has been collected in 
Sect. \ref{background}, the asymptotics of solutions of equation \eqref{basic} 
is studied in the regimes $\eta\to 0$ and $\eta\to\infty$. Theorems and 
proofs for the first of these cases are given in Sect.~\ref{asympsing}
(Theorems \ref{LinearSingularityExpansion} and \ref{LinearSingularityData})
 and for the second in Sect.~\ref{asymplatetime} 
(Theorem \ref{LinearExpandingThm}).
It is shown how all solutions can be parametrized by asymptotic data in either 
of these regimes. These are alternatives to the usual parametrization of 
solutions by Cauchy data. An interesting feature of the expanding direction 
$\eta\to\infty$ is that the main part of the asymptotic data is a solution of 
the flat space wave equation $W''=w\Delta W$. Many of these results can be 
extended to more general equations of state. This is the subject of 
Theorem \ref{GeneralSingularityExpansion} of Sect.~\ref{asympsing} 
(limit $\eta\to 0$) and Sect.~5. It is found that for equations of state 
with power law behaviour $p\sim\epsilon^{1+\sigma}$ at low density there
is a bifurcation with a fundamental change in the asymptotic behaviour at
$\sigma=\frac13$.

\section{Preliminaries}\label{background}

As outlined above, we study perturbations of FLRW cosmological models which 
are spatially flat and have $T^3$ spatial topology. The spacetime being
perturbed, which we refer to as the background, is described by a metric of 
the form 
	\begin{equation}\label{BackgroundMetric}
	a^2\left(-d\eta^2 + dx^2\right) 
	\end{equation}
on $(0,\infty)\times T^3$.  Here $dx^2$ indicates the flat metric on $T^3$ and 
the scale factor $a=a(\eta)$ is a non-decreasing function of the conformal 
time $\eta$.  We use $x$ to indicate points on $T^3$. The signature used 
here is the opposite of that used by Mukhanov \cite{mukhanov} but all the 
equations required in the following are unaffected by this change.

We make use of the perfect fluid matter model, described by the pressure $p$ 
and energy density $\epsilon$ of the fluid.  In order to specify the matter
model completely, one must provide an equation of state $p=f(\epsilon)$. Under 
this ansatz, the Einstein-Euler equations reduce to a coupled system of ODEs 
for $a$ and $\epsilon$:
	\begin{align}
	 a''& = \frac{4\pi G}{3}\left(\epsilon -3f(\epsilon)\right)a^3
\label{FriedmannEqn}
	 \\
	 \epsilon' &=-3\mathcal{H}\left(\epsilon + f(\epsilon)\right).
\label{ContinuityEqn}
	\end{align}
As mentioned in the introduction, $(\phantom{\Phi} )'$ indicates a derivative 
with respect to $\eta$.  Here $G$ is Newton's gravitational constant and 
$\mathcal{H}$ is the conformal Hubble parameter, given by 
$\mathcal{H} = a^{-1}a'$.  We note the following useful relation
(known as the Hamiltonian constraint)
	\begin{equation}\label{SecondFriedmann}
	 \mathcal{H}^2 = \frac{8\pi G}{3}a^2\epsilon.
	\end{equation}
For a linear equation of state $f(\epsilon) = w\epsilon$, solutions $a(\eta)$ 
of \eqref{FriedmannEqn} are explicitly given by
	\begin{equation}
	\frac{a(\eta)}{a(\eta_0)} = \left(\frac{\eta}{\eta_0}\right)^{2/(1+3w)},
	\end{equation}
for some arbitrarily fixed $\eta_0\in(0,\infty)$.  As the scale factor $a$ 
vanishes as $\eta\to 0$, the spacetime develops a curvature singularity in 
that limit, which is known as a ``big-bang'' type singularity and is viewed as 
being in the past of $\eta_0$.  Likewise the limit as $\eta\to\infty$ is 
referred to as ``late times'' as it corresponds to the distant future of $\eta_0$.  
Note that spacetimes described by these models are expanding, in the sense 
that the scale factor is an increasing function of $\eta$.
Note also that, since $\epsilon'$ is negative, large values of $\eta$ 
correspond to small values of $\epsilon$ and vice-versa.

We study behavior near the singularity and at late times for those 
perturbations to the metric \eqref{BackgroundMetric} which are of the type 
usually referred to as scalar perturbations. They satisfy evolution equations
obtained by linearizing the 
Einstein equations about the FLRW background.  For the perfect fluid matter 
model all such perturbations can be described, up to gauge freedom, by a 
single function $\Phi(\eta,x)$.  Using a certain gauge, the 
conformal-Newtonian gauge, the metric takes the form
	\begin{equation}
	a^2\left[-(1+2\lambda\Phi)d\eta^2 +(1-2\lambda\Phi)dx^2\right] 
	\end{equation}
up to an error which is quadratic in the expansion parameter $\lambda$.
The first order perturbation satisfies the linearized Einstein-Euler equations 
provided
	\begin{equation}\label{MainEquation}
	 \Phi'' + 3\left(1 + f'(\epsilon)\right)\mathcal{H}\Phi' 
	 +3 \left(f'(\epsilon)-\frac{f(\epsilon)}
{\epsilon}\right)\mathcal{H}^2\Phi 
	 - f'(\epsilon)\Delta\Phi=0,
	\end{equation}
where $\Delta$ is the Laplacian for the flat metric on $T^3$.
For a derivation of this equation we refer the reader to \S 7.2 of 
\cite{mukhanov}.
The corresponding perturbations to the energy density, denoted by 
$\delta\epsilon$, are determined by
	\begin{equation}\label{EnergyPerturbation}
	\delta\epsilon = \frac{1}{4\pi G a^2}
\left(-3\mathcal{H} \Phi' - 3\mathcal{H}^2\Phi + \Delta\Phi \right)
	\end{equation}
and thus can be computed once \eqref{MainEquation} is understood.

The quantity $f'(\epsilon)$ represents the square of the speed of sound for 
the fluid.  For physical reasons we require that $f'$ always take values in the 
interval $[0,1]$ i.e., that the speed of sound be real and not exceed 
the speed of light. A special case of particular interest is that of a linear 
equation of state $p=w\epsilon$. In this situation the speed of sound is 
constant and equation \eqref{MainEquation} reduces to \eqref{basic}. Before
the asymptotics of solutions of \eqref{MainEquation} can reasonably be 
studied a prerequisite is a theorem which guarantees global existence of
solutions on the interval $(0,\infty)$. In order to get this from the standard
theory of hyperbolic equations it is necessary to assume that $f'$ never
vanishes. In the following it is always assumed that this holds 
except in the special case of dust which is discussed separately.

Our analysis below relies on establishing a number of energy-type estimates 
for solutions to \eqref{MainEquation}.  As the coefficients of this linear 
equation depend only on $\eta$, any spatial derivative of $\Phi$ satisfies the 
same equation.  Thus any estimate we obtain for $\Phi$, $\Phi'$, or 
$\nabla\Phi$ (the gradient of $\Phi$ with respect to the flat metric on 
$T^3$) holds also for all spatial derivatives of those quantities.  One may 
then make use of the Sobolev embedding theorem in order to establish pointwise 
estimates.  We also make use of the Poincar\'e estimate which implies that 
quantities having zero (spatial) mean value are controlled in $L^2$ by the 
norm of their (spatial) gradient.

Each of these norms is defined on the $\eta$-constant ``spatial'' slices of 
$(0,\infty)\times T^3$ with respect to the flat ($\eta$-independent) metric 
induced on $T^3$ by viewing $T^3$ as a quotient of Euclidean space.  All 
integration on $T^3$ is done with respect to the corresponding volume element 
which we suppress in our notation.  We generally suppress dependence of 
functions on the spatial variable $x$, except in situations where the 
inclusion of such dependence provides additional clarity.  When necessary, we 
denote Cartesian coordinates on $T^3$ by $x = (x^i)$; the corresponding 
derivatives are denoted $\partial_i$.

\section{Asymptotics in the approach to the singularity}\label{asympsing}

The purpose of this section is to analyse the asymptotics of solutions of 
\eqref{basic} in the limit $\eta\to 0$ and to give some extensions of these
results to more general equations of state which need not be linear. Define 
$\nu=\frac12\left(\frac{5+3w}{1+3w}\right)$. Note that $\nu$ belongs to the 
interval $[1,5/2]$.

\begin{theorem}\label{LinearSingularityExpansion}
Let $\Phi(\eta)$ be a smooth solution of \eqref{basic} on 
$(0,\infty)\times T^3$. Then there are coefficients $\Phi_{k,l}$ with 
$k\ge -2\nu$ belonging to an increasing sequence of real numbers tending to 
infinity and $l\in\{0,1\}$, smooth functions on $T^3$, such that the formal 
series  $\sum_k(\Phi_{k,0}+\Phi_{k,1}\log\eta)\eta^k$ is asymptotic to 
$\Phi(\eta)$ 
in the limit $\eta\to 0$ in the sense of uniform convergence of the function
and its spatial derivatives of all orders. All coefficients can be expressed 
as linear combinations of $\Phi_{-2\nu,0}$, $\Phi_{0,0}$ and their spatial 
derivatives. If $\nu$ is not an integer then all coefficients with $l=1$ 
vanish. For any value of $w$ the coefficients $\Phi_{k,l}$ with $l=1$ and $k<0$ 
vanish.  

In more detail, $\Phi_{k,0}$ may only be non-zero when $k$ is of the 
form $-2\nu+2i$ or $2i$ for a non-negative integer $i$ while $\Phi_{k,1}$
may only be non-zero for $k$ of the form $2i$ with $i$ a non-negative
integer. These coefficients are related by the following equations: 
\begin{equation}\label{consist1}
k(k+2\nu)\Phi_{k,0}=w\Delta\Phi_{k-2,0}-(2k+2\nu)\Phi_{k,1}
\end{equation}
and
\begin{equation}\label{consist2}
k(k+2\nu)\Phi_{k,1}=w\Delta\Phi_{k-2,1}. 
\end{equation}
\end{theorem}

\begin{proof}
The basic tool which allows the solutions to be controlled
is provided by energy estimates. Let
\begin{equation}\label{energy}
E_1(\eta)=\frac12\int_{T^3} |\Phi'(\eta)|^2+w|\nabla\Phi(\eta)|^2.
\end{equation}
It satisfies the identity
\begin{equation}
\frac{d}{d\eta}\left[\eta^{2(2\nu+1)}E_1(\eta)\right]=(2\nu+1)\eta^{4\nu+1}\int_{T^3} 
w|\nabla\Phi(\eta)|^2.
\end{equation}
Since the right hand side is manifestly non-negative it can be 
concluded that if an initial time $\eta_0$ is given then 
$\eta^{2(2\nu+1)}E_1(\eta)$ is bounded 
for $\eta\le\eta_0$. Any spatial derivative of $\Phi$ satisfies the same
equation as $\Phi$. Thus corresponding bounds can be obtained for the $L^2$ 
norms of all spatial derivatives. Applying the Sobolev embedding theorem then 
provides pointwise bounds for $\Phi$ and its spatial derivatives of all
orders in the past of a fixed Cauchy surface. These estimates can now
be put back into the equation to obtain further information about the 
asymptotics. To do this it is convenient to write \eqref{basic} in the
form
\begin{equation}
\frac{d}{d\eta}\left[\eta^{2\nu+1}\Phi'(\eta)\right]=\eta^{2\nu+1}w\Delta\Phi(\eta).
\end{equation}
It can be deduced that
\begin{multline}\label{intformula}
\Phi'(\eta)=\eta^{-2\nu-1}\Big[\eta_0^{2\nu+1}\Phi'(\eta_0)
\\ 
-w\int_0^{\eta_0}\zeta^{2\nu+1}\Delta\Phi(\zeta)d\zeta
+w\int_0^\eta\zeta^{2\nu+1}\Delta\Phi(\zeta)
d\zeta\Big]
.\end{multline}  
The bounds already obtained guarantee the convergence of the integrals.
This formula allows the asymptotic expansions to be derived inductively.
Using the fact that the second integral is $O(\eta^2)$ already gives a 
one-term expansion for $\Phi'$ and this can be integrated to give
a one-term expansion for $\Phi$. Analogous expansions can be obtained for
all spatial derivatives of $\Phi$ in the same way using the corresponding
spatial derivatives of \eqref{intformula}. When an asymptotic expansion 
with a finite number of explicit terms is substituted into the right hand
side of \eqref{intformula} an expansion for $\Phi'$ (and thus by integration
for $\Phi$) with additional explicit terms is obtained. If the last explicit
term in the input is a multiple of $\eta^p$ with $p<-2$ then there is one new 
term in the output and it is a multiple of $\eta^{p+2}$. If the last explicit 
term is a multiple of $\eta^{-2}$ there is one new term and it is a multiple 
of $\log \eta$. If the last explicit term is a multiple of $\log \eta$ then 
there are two new terms, one a multiple of $\eta^2\log \eta$ and one a 
constant. If the last explicit term is $\eta^p$ or $\eta^p\log \eta$ with 
$p>-2$ then there is one new term and it is a multiple of $\eta^{p+2}$ or 
$\eta^{p+2}\log \eta$ respectively. These statements rely on the fact that
when any of the terms in the asymptotic expansion is substituted into the
last integral in \eqref{intformula} the power $-1$ never arises. These 
remarks suffice to prove the first part of the theorem. The resulting series 
is by construction a formal series solution of the original equation. 
Comparing coefficients gives the rest of the theorem.
\end{proof}

Note that the only two values of $w$ in the range of interest where 
logarithmic terms occur in the expansions of the theorem are $w=\frac19$
and $w=1$ corresponding to $\nu=2$ and $\nu=1$ respectively. The two cases
of most physical interest, $w=0$ (dust) and $w=\frac13$ (radiation), are 
free of logarithms. In the case $w=0$ most of the expansion coefficients 
vanish and the two non-vanishing terms define an explicit solution which is 
a linear combination of two powers of $\eta$.

The relative density perturbation is given by
\begin{equation}\label{densitypert}
\frac{\delta\epsilon}{\epsilon}=
-2\Phi-2{\cal H}^{-1}\Phi'+\frac23{\cal H}^{-2}\Delta\Phi.
\end{equation}
Now ${\cal H}=\frac{2}{(1+3w)\eta}$. Substituting this relation and the 
asymptotic expansion for $\Phi$ into the expression for the density 
perturbation gives: 
\begin{multline}\label{densitypert2}
\frac{\delta\epsilon}{\epsilon}=\sum_k\left[-(k(1+3w)+2)\Phi_{k,0}
-(1+3w)\Phi_{k,1}+\frac16(1+3w)^2\Delta\Phi_{k-2,0}\right. \\
\left.+(-(k(1+3w)+2)\Phi_{k,1}+\frac16(1+3w)^2\Delta\Phi_{k-2,1}
\log\eta\right]\eta^k
\end{multline}

The relations in Theorem \ref{LinearSingularityExpansion} place no 
restrictions on the coefficients
$\Phi_{-2\nu,0}$ and $\Phi_{0,0}$ and so it is natural to ask if these can
be prescribed freely. In other words, if two smooth functions on $T^3$
are given, is there a smooth solution of the equations in
whose asymptotic expansion for $\eta\to 0$ precisely these functions
occur as the coefficients $\Phi_{-2\nu,0}$ and $\Phi_{0,0}$? The next 
theorem answers this question in the affirmative. Since the proof 
is closely analogous to arguments which are already in the literature
it will only be sketched.

\begin{theorem}\label{LinearSingularityData}
Let $\Psi_1$ and $\Psi_2$ be smooth functions on $T^3$.
Then there exists a unique solution of \eqref{basic} of the type considered
in Theorem \ref{LinearSingularityExpansion} with $\Phi_{-2\nu,0}=\Psi_1$ and 
$\Phi_{0,0}=\Psi_2$.
\end{theorem}

\begin{proof}[Proof (sketch)]
The proof of this theorem uses Fuchsian techniques. It 
implements the strategy applied in \cite{rendall00} to prove theorems on the 
existence of solutions of the vacuum Einstein equations belonging to the Gowdy 
class with prescribed singularity structure. In the present situation some 
simplifications arise in comparison to the argument for Gowdy due to the fact
that the equation being considered is linear. The procedure is to first
treat the case of analytic data and then use the resulting analytic 
solutions to handle the smooth case.  To reduce the equation to
Fuchsian form the following new variables are introduced. First 
define a function $v(\eta,x)$ by the relation
\begin{equation}\label{expansion}
\Phi(\eta)=\Psi_1\eta^{-2\nu}+\sum_{-2\nu<k<0}\Phi_{k,0}\eta^k+\Phi_{0,1}\log\eta
+\Psi_2+v(\eta).
\end{equation}
Here it is assumed that the consistency relations \eqref{consist1} hold for
$-2\nu\le k\le 0$. As a consequence of these relations and the original 
equation, $v$ satisfies
\begin{equation}\label{firstfuch}
v''+\frac{2\nu+1}{\eta}v'-w\Delta v=w\Delta\Phi_{0,1}\log\eta+w\Delta\Psi_2
+w\sum_{0<k<2}\Delta\Phi_{k,0}\eta^k
\end{equation}
Note that the last sum will contain one non-vanishing term for $\nu$ not an 
integer and none for $\nu$ an integer. Denote the right hand side of 
\eqref{firstfuch} by $Q$. This equation can be reduced to a first order 
system by introducing new variables $v^0=\eta v'$ and $v^i=\eta\partial_i v$. 
Let $V$ be the vector-valued unknown with components $(v,v^0,v^i)$. Then the
first order system is
\begin{equation}
\eta\partial_\eta V+NV=\eta^\zeta f(\eta,V,DV)
\end{equation}
where
\begin{equation}
N=\left[\begin{array}{ccc}
0 & -1 & 0\\0 & 2\nu & 0\\ 0 & 0 & 0\end{array}\right]
\ \ \ {\rm and}\ \ \ 
f=\left[\begin{array}{c}
0\\ 
\eta^{1-\zeta}w\partial_i v^i+\eta^{2-\zeta}Q\\
\eta^{1-\zeta}\partial_i(v^0+v)\end{array}\right].
\end{equation}

Here $\zeta$ is any positive real number less than one and $DV$ denotes 
the collection of spatial derivatives of $V$. 
It will be shown that this equation has a unique solution $v$ which converges 
to zero as $\eta\to 0$. Initially we assume that the functions 
$\Psi_1$ and $\Psi_2$ are analytic. Then results  
proved in \cite{kr} can be applied. See also section 4 of \cite{ar} for some 
further information on these ideas. One of the hypotheses required
is that $f$ is regular in the analytic sense defined there. What this means 
is that $f$ and all its derivatives with respect to any argument other 
than $\eta$ are real analytic for $\eta>0$ and extend continuously  to 
$\eta=0$. The other hypothesis is that the matrix exponential
$\sigma^N$ should be uniformly bounded for all $0<\sigma<1$. This follows 
from the fact that $N$ is diagonalizable with non-negative eigenvalues. 

To extend this result to the smooth case more work is necessary. The basic
idea is to approximate the smooth functions $\Psi_1$ and $\Psi_2$ by
sequences of analytic functions $(\Psi_1)_n$ and $(\Psi_2)_n$, apply the
analytic existence theorem just discussed to get a sequence of solutions
$V_n$ of the Fuchsian system and then show that $V_n$ tends to a limit
$V$ as $n\to\infty$. The function $V$ is then the solution of the problem
with smooth data. To show the convergence of $V_n$ suitable estimates
are required and in order to obtain these the Fuchsian equation is written 
in an alternative form which is symmetric hyperbolic. This rewriting is
only possible for $w\ne 0$ but for $w=0$ the system, being an ODE, is 
already symmetric hyperbolic and so the extra step is not required.
In general a simplification of the system is achieved by introducing a new 
time variable by $t=\eta^\zeta$ and rescaling $f$ by a factor $\zeta^{-1}$. 
Then the system can be written as 
\begin{equation}
tA^0\partial_t V+tA^j\partial_jV+MV=tg(t,V,DV)
\end{equation}
where
\begin{equation}
M=\left[\begin{array}{ccc}
0 & -1 & 0\\0 & \frac{2\nu}{\zeta w} & 0\\ 0 & 0 & 
-\frac{1+\zeta}{\zeta}I\end{array}\right]
\ \ \ , \ \ \ 
g=\left[\begin{array}{c}
0\\ 
t^{\frac{2-\zeta}{\zeta}}Q\\
0\end{array}\right]
\end{equation} 
and the other coefficient matrices are given by $A^0={\rm diag}(1,\frac1w,I)$ 
and
\begin{equation}
A^j=\left[\begin{array}{ccc}
0 & 0 & 0\\
0 & 0 & -\zeta^{-1}t^{\frac{1-\zeta}{\zeta}}e_j\\
0 & -\zeta^{-1}t^{\frac{1-\zeta}{\zeta}}e_j & 0
\end{array}\right]
\end{equation}
with $e_j$ the $j$th standard basis vector in $R^3$. This is a symmetric
hyperbolic system. A disadvantage is that in passing from $N$ to $M$
positivity is lost. 

The fact that $M$ has a negative eigenvalue can be overcome by 
subtracting an approximate solution from $v$ to obtain a new unknown.
Expressing the equation in terms of the new unknown leads to a system
which is similar to that for $v$ but with $M$ replaced by $M+nI$ for an
integer $n$. For $n$ sufficiently large this means that the replacement
for $M$ is positive definite. With this choice the system is both 
in Fuchsian form and symmetric hyperbolic. The necessary approximate
solution can be taken to be a formal solution of sufficently high order  
as introduced in section 2 of \cite{rendall00}. The fact that the system 
is symmetric hyperbolic leads to energy estimates which can be used to 
prove the convergence of the sequence of analytic solutions to a solution
corresponding to the smooth initial data, thus completing the proof of 
the existence part of the theorem. Uniqueness can be proved using an
energy estimate as has been worked out in \cite{rendall00}.
\end{proof}

It would presumably be possible to extend the above results to the case that 
the data are only assumed to belong to a suitable Sobolev space. An 
alternative approach to doing so would be try to apply ideas in the paper
\cite{kichenassamy} of Kichenassamy.

The proofs just presented have been strongly influenced by work on 
Gowdy spacetimes. For a special class of these, the polarized
Gowdy spacetimes, the basic field equation is $P_{tt}+t^{-1}P_t=P_{xx}$.
Evidently this is closely related to \eqref{basic} although they
are not identical for any choice of $w$, even if attention is 
restricted to solutions of \eqref{basic} depending on only one space 
variable. The energy arguments above were inspired by those applied to the 
polarized Gowdy equation in \cite{im}. The following analogue 
of Theorem 2 is a special case of a result in \cite{rendall00}. If smooth
periodic functions $k(x)$ and $\omega(x)$ are given with $k$ everywhere
positive there is a smooth solution of the polarized Gowdy equations
which satisfies  
\begin{equation}
P(t,x)=k(x)\log t+\omega (x)+o(1)
\end{equation}
as $t\to 0$. It is plausible that the positivity restriction on $k$, while
very important for general (non-polarized) Gowdy spacetimes, should be 
irrelevant in the polarized case. It turns out that following the arguments
used above to analyse \eqref{basic} allows this intuition to be proved
correct.

One way of attempting to reduce the polarized Gowdy equation to Fuchsian
form is to mimic \eqref{expansion} and write $P=k\log t +\omega+v$. This 
fails because the analogue of the matrix $N$ has $\nu$ replaced by zero.
Thus the matrix has all eigenvalues zero and includes a non-trivial Jordan 
block. To access the Fuchsian theory in the analytic case the
expansion for $P$ may be replaced by 
\begin{equation}
P=k\log t +\omega+t^\delta v
\end{equation} 
for a small positive $\delta$. With this modification the reduction procedure 
applied to \eqref{basic} gives a Fuchsian system. It can be concluded that
$k$ and $\omega$ can be prescribed in the case that they are analytic.
Once this has been achieved the smooth case can be handled just as in
the proof of Theorem 2.

It will now be shown that some of the results which have been proved for a
linear equation of state can be extended to more general equations of state.
In the discussion which follows it will be convenient to exclude the case
of a linear equation of state which has been treated already. This in 
particular excludes dust so that by our general assumptions $f'$ never
vanishes. In this case we consider solutions to \eqref{MainEquation} rather than 
\eqref{basic}.
Choose an initial time $\eta_0$ and for a given background solution let
$\epsilon (\eta_0)=\epsilon_0$. 
From the condition that $f'(\epsilon)\leq 1$ it follows that
\begin{equation}
\Lambda:=\sup_{(\epsilon_0,\infty)}\left|f'(\epsilon)-\frac{f(\epsilon)}
{\epsilon}\right|^{1/2}
\end{equation}
is strictly positive and finite. It will be assumed in addition that the 
equation of state satisfies the condition
\begin{equation}\label{secondderiv}
\sup_{(\epsilon_0,\infty)}\left|\left(\frac{\epsilon+f(\epsilon)}
{f'(\epsilon)}\right)\frac{d^2 f}{d\epsilon^2}\right|<\infty.
\end{equation}
Using the fact that $\Lambda>0$ it follows that there exists a positive 
number $\lambda$ satisfying the following three inequalities:
\begin{equation}\label{ineq1}
2\lambda\frac{df}{d\epsilon}\ge 3(\epsilon+f(\epsilon))
\frac{d^2 f}{d\epsilon^2},
\end{equation}

\begin{multline}\label{ineq2}
4\lambda^2-2\left[6\left(1+f'(\epsilon)\right)+\left(1+\frac{3f(\epsilon)}
{\epsilon}
\right)\right]\lambda  \\
+6\left(1+f'(\epsilon)\right)\left(1+\frac{3f(\epsilon)}{\epsilon}\right) 
- \Lambda^{-2}\left|\Lambda^{2}-3\left(f'(\epsilon)-\frac{f(\epsilon)}
{\epsilon}\right)
\right|^2\ge 0
\end{multline}
and

\begin{equation}\label{ineq3}
\lambda\ge 3(1+f'(\epsilon)).
\end{equation}
That \eqref{ineq1} can be satisfied follows from \eqref{secondderiv}. The 
fact that $f'(\epsilon)$ and $f(\epsilon)/\epsilon$ are bounded means 
that the first term in the expression on the left hand side of \eqref{ineq2} 
dominates the other terms for $\lambda$ sufficiently large and so the second
condition on $\lambda$ can also be satisfied. The constant $\lambda$ can be
chosen to satisfy \eqref{ineq3} since the right hand side of that inequality
is bounded. Note for comparison that for a linear equation 
of state $\Lambda=0$. In that case $\lambda$ can be taken to be 
the larger root of the expression obtained from the left hand side of 
\eqref{ineq2} by omitting the term containing $\Lambda$. This root is $3(1+w)$.
Define the following generalization of the energy functional \eqref{energy}:
\begin{equation}
E_2(\eta)=\frac12\int_{T^3} |\Phi'(\eta)|^2+f'(\epsilon)|\nabla\Phi(\eta)|^2
+\Lambda {\cal H}^2|\Phi(\eta)|^2.
\end{equation}
(Note that we suppress the dependence of $\epsilon$ and $\mathcal{H}$ on 
$\eta$.)
A computation shows that if $a$ denotes the scale factor, then due to the
inequalities \eqref{ineq1}-\eqref{ineq3}
\begin{equation}\label{energynonlin}
\frac{d}{d\eta}\left[a^{2\lambda}E_2(\eta)\right]\ge 0.
\end{equation}
In more detail, computing the time derivative of $a^{2\lambda}E_2$ and using 
equation \eqref{MainEquation} along with the equations satisfied by the 
background
quantities $\epsilon$ and ${\cal H}$ gives an integral where the 
the integrand is a sum of terms each of which has a factor $\Phi^2$,
$|\Phi'|^2$, $\Phi\Phi'$ or $|\nabla\Phi|^2$. The aim is to show that
the sum of these terms is non-negative. To do this it is first assumed that
the coefficient of $|\nabla\Phi|^2$ is non-negative. This leads to the 
condition \eqref{ineq1}. Next it is shown that the quadratic form in
$\Phi$ and $\Phi'$ is positive semidefinite. This can be done by using 
the inequality 
\begin{equation}
|\Lambda^{} {\cal H}\Phi\Phi'|\le 
\frac{\delta}{2}\Lambda^2 {\cal H}^2\Phi^2 +\frac{1}{2\delta}(\Phi')^2,
\end{equation}
which holds for any $\delta>0$, to estimate the quadratic form from below by 
the following sum of a term containing $\Phi^2$ and one containing $|\Phi'|^2$:
\begin{multline}
\frac12\Lambda^2{\cal H}^2\left[2\lambda-\left(1+\frac{3f(\epsilon)}
{\epsilon}\right)
-\delta \Lambda^{-1} \left|\Lambda^2-3\left(f'(\epsilon) -\frac{f(\epsilon)}
{\epsilon}\right)\right|\right]\Phi^2
 \\
+\frac12\left[2\lambda-6(1+f'(\epsilon))
-\frac{\Lambda^{-1}}{\delta}\left|\Lambda^2-{3}\left(f'(\epsilon)
-\frac{f(\epsilon)}{\epsilon}\right)\right|\right]|\Phi'|^2
\end{multline}
It 
remains to ensure that the coefficients of these terms are non-negative and 
this follows from \eqref{ineq2} and \eqref{ineq3}, choosing $\delta$
sufficiently small. It can be concluded from 
\eqref{energynonlin}
that $E_2(\eta)=O(a(\eta)^{-2\lambda})$ as $\eta\to 0$. As in the case
of a linear equation of state, corresponding estimates hold for spatial 
derivatives and pointwise estimates follow by Sobolev embedding. An integral
formula for $\Phi'$ can be obtained as in the case of a linear equation of
state. It reads (with some arguments suppressed; recall $\epsilon$ and 
$\mathcal{H}$ depend on $\eta$)
\begin{multline}\label{intformula2}
\Phi'(\eta)=\left(f(\epsilon)+\epsilon\right)
\left[\frac{\Phi'(\eta_0)}{f(\epsilon_0)+\epsilon_0}\right.
 \\
\left.-\int_{\eta_0}^\eta\frac{1}{f(\epsilon)+\epsilon}
\left(f'(\epsilon)\Delta\Phi+3\left(f'(\epsilon)-\frac{f(\epsilon)}
{\epsilon}\right)\mathcal{H}^2\Phi\right)
\right]
\end{multline}
If no further assumptions are made on the equation of state then using
the known boundedness statements and repeatedly substituting into the 
right hand side of \eqref{intformula2} would lead to unwieldy expressions 
involving iterated integrals. Simpler results can be obtained if it is assumed
that in the limit $\epsilon\to\infty$ the function $f$ is linear in leading
order with lower powers as corrections. In other words for this assume that
$f$ admits an asymptotic expansion of the form
\begin{equation}\label{eqnofstate}
f(\epsilon)\sim w\epsilon+\sum_{j=1}^\infty f_j\epsilon^{a_j}
\end{equation} 
as $\epsilon\to\infty$. Here the $f_j$ are constants while $\{a_j\}$ is a 
decreasing sequence of real numbers all of which are less than one and which 
tend to $-\infty$ as $j\to\infty$. Assume further that the relation obtained
by differentiating this expansion term by term any number of times is also a 
valid asymptotic expansion. To have a concrete example, consider the 
polytropic equation of state which is given parametrically by the relations
\begin{equation}\label{polytropic}
\epsilon=m+Knm^{\frac{n+1}{n}},\qquad p=Km^{\frac{n+1}{n}}
\end{equation}
with constants $K$ and $n$ satisfying $0<K<1$ and $n>1$. In this case
the asymptotic expansion is of the form
\begin{equation}\label{stateexpansion}
f(\epsilon)=n^{-1}\epsilon+n^{-1}(Kn)^{\frac{1}{n+1}}\epsilon^{\frac{n}{n+1}}
+\ldots
\end{equation}
Returning to the more general case \eqref{eqnofstate}, define a quantity
$m$ by
\begin{equation}\label{massdensity}
m(\epsilon)=\exp\left\{\int_1^\epsilon (\xi+f(\xi))^{-1}d\xi \right\}
\end{equation} 
Substituting the asymptotic expression \eqref{eqnofstate} into 
\eqref{massdensity} gives a corresponding asymptotic expansion for the 
function $m(\epsilon)$ as a sum of powers of $\epsilon$ with the leading 
term being proportional to $\epsilon^{\frac{1}{w+1}}$. It follows from the 
continuity equation \eqref{ContinuityEqn} for the fluid that $m$ is 
proportional to $a^{-3}$. This 
leads to an asymptotic expansion for $\epsilon$ in terms of $a$. The equation 
\eqref{SecondFriedmann}
 implies that
$a'=\sqrt{8\pi G/3}\epsilon^{1/2}a^2$; substituting for
$\epsilon$ in terms of $a$ gives rise to a relation which can be 
integrated to give an asymptotic expansion for $a$ in terms of $\eta$
in the limit $\eta\to 0$. The leading term is proportional to 
$\eta^{\frac{2}{3w+1}}$. Substituting this back in leads to an asymptotic
expansion for $a'$ from which an asymptotic expansion for ${\cal H}$ can be
obtained. An asymptotic expression for $\epsilon$ in terms of $\eta$ can also
be derived. Thus in the end there are expansions for all the important
quantities in the background solution in terms of $\eta$. In all cases 
the leading term in the expansion agrees with that in the case of a linear 
equation of state. The result is an integral equation which can be 
written in the form
\begin{eqnarray}\label{intformula3}
&&\Phi' (\eta)=h_1 (\eta)\left[C-
\int_0^{\eta_0} h_2(\zeta)\Delta\Phi (\zeta)
-h_3(\zeta)\Phi (\zeta)d\zeta\right.\nonumber    \\
&&\left.+\int_0^\eta h_2(\zeta)\Delta\Phi (\zeta)
+h_3(\zeta)\Phi (\zeta)d\zeta\right]
\end{eqnarray} 
where $C$ is a constant depending only on the data at time $\eta_0$
and asymptotic expansions are available for the 
functions $h_1$, $h_2$ and $h_3$. The leading terms in $h_1$ and $h_2$
are constant multiples of the corresponding powers of $\eta$ for a linear 
equation of state. To see the leading order behaviour of $h_3$ recall that 
in \eqref{intformula3}, the coefficient of $\Phi$ is
\begin{equation}\label{evolH}
3\left(f'(\epsilon)-\frac{f(\epsilon)}{\epsilon}\right){\cal H}^2.
\end{equation}
Hence if $f_j$ is the first non-vanishing coefficient in the expansion
\eqref{stateexpansion} then the leading order power in $h_3$ is less
than that in $h_2$ by $\alpha=2-\frac{6(1+w)(1-a_j)}{1+3w}$. To obtain 
estimates close to $\eta=0$ the estimate for the
energy can be applied starting from $\eta$ very small. In other words,
$\epsilon_0$ can be chosen as large as desired. Then all the coefficients in 
the left hand side of \eqref{ineq2} not involving $\Lambda$ are as close as 
desired to those for the corresponding linear equation of state.
Since $\Lambda$ is arbitrarily small, the coefficient involving $\Lambda$ is 
also arbitrarily small. It follows that $\lambda$ can be chosen to
have any value strictly greater than $3(1+w)$. Hence $E$ can be bounded
by any power greater than the power in the corresponding linear case.  
This is enough to proceed as in the proof of Theorem 1 to 
obtain an asymptotic expansion for $\Phi$ where each invidual term is 
a constant multiple of an expression of the form $\eta^k(\log\eta)^l$
with $l=0$ or $l=1$ and the leading term is just as in Theorem 1 with the 
corresponding value of $w$. The key thing that makes this work is that
$\alpha<2$ so that no logarithms are generated when evaluating the 
integral in \eqref{intformula3} in the course of the iteration.
The results of this discussion can be summed up as follows.

\begin{theorem}\label{GeneralSingularityExpansion}
Let $\Phi$ be a smooth solution of \eqref{MainEquation} on 
$(0,\infty)\times T^3$. Suppose that the equation of state has an asymptotic
expansion of the form \eqref{eqnofstate}. Then there are coefficients 
$\Phi_{k,l}$ 
with $k\ge -2\nu$ belonging to an increasing sequence of real numbers tending 
to infinity and $l\in\{0,1\}$, smooth functions on $T^3$, such that the formal 
series $\sum_k\Phi_{k,l}(\log\eta)^l\eta^k$ is asymptotic to $\Phi$ in the 
limit $\eta\to 0$. All coefficients in the expansion are determined uniquely
by $\Phi_{-2\nu,0}$ and $\Phi_{0,0}$.
\end{theorem}

\section{Late-time asymptotics for a linear equation of state}
\label{asymplatetime}

In this section information is obtained about the asymptotics of 
solutions of equation \eqref{basic} in the limit $\eta\to\infty$; 
some extensions of these results to more general equations of state are 
derived in Sect. \ref{asymplatetimegen}. Once again energy estimates play 
a fundamental role. 
In this case it is convenient to treat homogeneous solutions separately. 
By a homogeneous solution we mean one which does not depend on the spatial
coordinates. These can be characterized as the solutions whose initial
data on a given spacelike hypersurface do not depend on the spatial
coordinates. For this class of solutions equation \eqref{basic} can be 
solved explicitly with the result that $\Phi=A+B\eta^{-2\nu}$ for constants 
$A$ and $B$. A general solution can be written as the sum of a 
homogeneous solution and a solution such that $\Phi$ has zero mean 
on any hypersurface of constant conformal time. Call solutions of the latter 
type zero-mean solutions. Then in order to determine the late-time asymptotics
for general solutions it suffices to do so for zero-mean solutions. In
this case define $\psi(\eta)=\eta^{\nu+\frac12}\Phi(\eta)$. Then $\psi$ 
satisfies the equation
\begin{equation}\label{evolpsi}
\psi''=w\Delta\psi+\left(\nu^2-\frac14\right)\eta^{-2}\psi
\end{equation}
Define an energy by
\begin{equation}\label{energy2}
E_3(\eta)=\frac12\int_{T^3} |\psi'(\eta)|^2+w|\nabla\psi(\eta)|^2.
\end{equation}
Then
\begin{equation}
E_3'(\eta)=2\left(\nu^2-\frac14\right)\eta^{-2}\int_{T^3} 
\psi(\eta)\psi'(\eta).
\end{equation}
The integral on the right hand side of this equation can be bounded, using
the Cauchy-Schwarz inequality, in terms of the $L^2$ norms of $\psi'$ and
$\psi$. The first of these can be bounded in terms of the energy and due
to the fact that the mean value of $\psi$ is zero, the same is true of the
second. Thus $E_3'(\eta)\le C\eta^{-2} E_3(\eta)$ for a constant 
$C$. By 
Gronwall's inequality it follows that $E_3$ is globally bounded in the 
future. These 
arguments apply equally well to spatial derivatives of $\psi$ of any order.
By the Sobolev embedding theorem it can be concluded that $\psi$ and its
spatial derivatives of any order are bounded. The energy bounds and the 
basic equation then imply that all spacetime derivatives of any order are
uniformly bounded in time.

Let $\eta_j$ be a sequence of times tending to infinity and consider the
translates defined by $\psi_j(\eta)=\psi(\eta+\eta_j)$. The sequence 
$\psi_j$ satisfies uniform $C^\infty$ bounds. Consider the restriction
of this sequence to an interval $[\eta_0,\eta_1]$. By the Arzel\`a-Ascoli
theorem the sequence of restrictions has a uniformly convergent subsequence.
By passing to further subsequences and diagonalization it can be shown
that $\psi$ and its spacetime derivatives of all orders converge uniformly
on compact subsets to a limit $W$. Passing to the limit in the evolution 
equation for $\psi$ along one of these sequences shows that $W$ satisfies 
the flat-space wave equation $W''=w\Delta W$.  Note that \emph{a priori} the 
function $W$ could depend on the sequence of times chosen. This issue is 
examined more closely below.

Given a smooth solution of \eqref{evolpsi} it is possible to 
do a Fourier transform in space to get the equation
\begin{equation}\label{mode}  
\hat\psi''=-w|k|^2\hat\psi+\left(\nu^2-\frac14\right)\eta^{-2}\hat\psi
\end{equation}
which is referred to below as the mode equation. Here $k$ is a vector. The
restriction to zero-mean solutions implies that the case $k=0$ of
\eqref{mode} can be ignored.

\begin{lemma}\label{ModeLemma}
Any solution $\hat\phi$ of equation \eqref{mode} has an asymptotic 
expansion of the form
\begin{equation}
\hat\phi (\eta)=\bar W_k\cos(\sqrt{w}|k|(\eta-\bar\eta_k))+O(\eta^{-1}),
\end{equation}
for constants $\bar\eta_k$ and $\bar W_k$, in the limit $\eta\to\infty$.
\end{lemma}
\begin{proof}
To prove the lemma it is convenient to introduce polar coordinates
associated to the variables $\hat\psi$ and $\frac{1}{\sqrt{w}|k|}\hat\psi'$. 
Thus 
$\hat\psi=r\cos\theta$ and $\frac{1}{\sqrt{w}|k|}\hat\psi'=r\sin\theta$. This 
leads 
to the equations:
\begin{eqnarray}
r'&=&\frac{1}{\sqrt{w}|k|}\left(\nu^2-\frac14\right)r\eta^{-2}
\sin\theta\cos\theta
\label{evolr} \\ 
\theta'&=&-\sqrt{w}|k|+\frac{1}{\sqrt{w}|k|}\left(\nu^2-\frac14\right)
\eta^{-2}\cos^2\theta
\label{evoltheta}
\end{eqnarray}
It follows from \eqref{evoltheta} that
\begin{equation}\label{asymptheta}
\theta (\eta)=-\sqrt{w}|k|(\eta-\bar\eta_k)+O(\eta^{-1})
\end{equation}
for a constant $\bar\eta_k$. From \eqref{evolr} it follows that 
\begin{equation}\label{asympr}
r(\eta)=\bar W_k (1+O(\eta^{-1}))
\end{equation}
for a constant $\bar W_k$. As a consequence of \eqref{asymptheta} we have
\begin{equation}
\cos (\eta(\theta))=\cos(\sqrt{w}|k|(\eta-\bar\eta_k))+O(\eta^{-1}).
\end{equation} 
Together with \eqref{asympr} this gives the conclusion of the lemma.
\end{proof}

Consider a zero-mean solution of the type considered before. Let a function 
$W$ be defined by taking the sequence $\eta_j$ used above to consist of 
integer multiples of $2\pi$. We now show that the function $\psi-W$ 
tends to zero as $\eta\to\infty$. In order to do this it suffices to show 
that it does so along a subsequence of an arbitrary sequence of values 
$\zeta_j$ of $\eta$ tending to infinity. By passing to a subsequence as 
before it can be arranged that the translates by the amounts $\zeta_j$ 
converge uniformly on compact subsets 
as $j\to\infty$. Call the limit $Y$. The aim is to prove that $Y=0$. If
not there must be some mode $\hat Y$ which is non-zero. It can be obtained 
as the limit of some $\hat\psi-\hat W$. From Lemma \ref{ModeLemma} it can be 
seen that
$\hat W=\bar W_k\cos (\sqrt{w}|k|(\eta-\eta_k))$. Hence 
$\hat\psi-\hat W=O(\eta^{-1})$
and so $\hat Y=0$, a contradiction. Convergence of derivatives can be 
obtained in a corresponding way. Thus any solution can be written as
$\Phi(\eta,x)=\eta^{-\nu-\frac12}(W(\eta,x)+o(1))$. A similar result for
the polarized Gowdy equation  with a sharper estimate on the error term
was proved in \cite{jurke}.

A late-time asymptotic expansion has now been derived which involves a 
solution $W$ of the flat-space wave equation. Comparing with the results
on parametrizing solutions by the coefficients in an asymptotic expansion
near the singularity it is natural to ask if the function $W$ can be
prescribed freely. It will now be shown that this is the case by
following the proof of an analogous result for the polarized Gowdy 
equation due to Ringstr\"om \cite{ringstrom05}. Write an arbitrary 
zero-mean solution in the form
\begin{equation}
\Phi(\eta,x)=\eta^{-\nu-\frac12}W(\eta,x)+\omega (\eta,x).
\end{equation}  
Then $\omega$ satisfies the equation
\begin{equation}
\omega''+\eta^{-1}\omega'-w\Delta\omega=\left(\nu^2-\frac14\right)
\eta^{-\nu-\frac52}W.
\end{equation}
Define
\begin{equation}
H(\eta)=\frac12 \int_{T^3} |\omega'(\eta)|^2+w|\nabla\omega(\eta)|^2  
\end{equation}
and
\begin{equation}
\Gamma(\eta)=\frac{1}{2\eta}\int_{T^3} \omega(\eta)\omega'(\eta).
\end{equation} 
The aim is to study late times and attention will be restricted to the
region where $\eta\ge w^{-1}$. At this point it is necessary to assume 
that $w>0$. The following inequalities show the equivalence of $H$ and
$H+\Gamma$ as norms of $(\omega',\nabla\omega)$:
\begin{equation}
|\Gamma(\eta)|\le\frac{1}{2w\eta}H(\eta)\ \ \ ,\ \ \frac12 H\le H
+\Gamma\le\frac32 H.
\end{equation}
Now
\begin{multline}
\frac{d}{d\eta}\left[H+\Gamma\right]=-\frac{1}{\eta}(H+\Gamma)-\frac{4\nu+3}{2\eta}\Gamma
 \\
+\left(\nu^2-\frac14\right)\eta^{-\nu-\frac52}\int_{T^3}\omega'W
+\frac12\left(\nu^2-\frac14\right)\eta^{-\nu-\frac72}\int_{T^3}\omega W.
\end{multline}
Using the equivalence of $H+\Gamma$ and $H$ this can be used to derive
the following differential inequality
\begin{multline}
\frac{d}{d\eta}\left[H+\Gamma\right]\ge -\left(\frac{1}{\eta}+\frac{4\nu+3}{2w\eta^2}\right)
\left(H+\Gamma\right)
\\
-\eta^{-\nu-\frac52}\|W \|_{L^2}\left(\nu^2-\frac14\right)
\left(\frac{2 +\sqrt{w}}{\sqrt{2}}\right)\left(H+\Gamma\right)^{1/2}
\end{multline}
By analogy with equation (16) of \cite{ringstrom05} define
\begin{equation}
E_4(\eta)=\eta e^{\frac{4\nu+3}{2\eta w}}(H(\eta)+\Gamma(\eta)).
\end{equation}
This quantity satisfies an inequality of the form
\begin{equation}
E_4'(\eta)\ge -C\eta^{-\nu-2}\|W(\eta)\|_{L^2}E_4(\eta)^{1/2}
\end{equation}
for a positive constant $C$ depending on $w$. Since $\eta^{-\nu-2}$ is 
integrable at infinity this inequality can be used in just the same way 
as the corresponding inequality in \cite{ringstrom05}. In this way it 
can be proved that given a solution $W$ of the flat space wave equation
there is a corresponding solution $\Phi$ of \eqref{basic}. It follows from 
the proof that $E_4(\eta)=O(\eta^{-2\nu-2})$. Hence $H(\eta)=O(\eta^{-2\nu-3})$
and the solution decays like $\eta^{-3/2-\nu}$.

The information obtained concerning the asymptotics of the solutions
constructed starting from a solution $W$ of the wave equation is stronger
that what was proved about general solutions of \eqref{basic} up to this 
point. This can be improved on as follows. Given a solution $\Phi$ of
\eqref{basic} a solution $W$ of the flat space wave equation is obtained.
From there a solution $\tilde\Phi$ of \eqref{basic} is obtained with stronger 
information on the asymptotics. The aim is now to show that $\tilde\Phi=\Phi$.
To do this it is enough to show that each Fourier mode agrees. This means
showing that a solution $\hat\psi$ of \eqref{mode} vanishes if it tends to 
zero as $\eta\to\infty$. That the latter statement holds follows easily 
from \eqref{evolr}. What has been proved can be summed up in the following
theorem.

\begin{theorem}\label{LinearExpandingThm}
Let $\Phi$ be a global smooth solution of \eqref{basic}. Then
there exist constants $A$ and $B$ and a smooth solution $W$ of the equation 
$W''=w\Delta W$ with zero spatial average such that
\begin{equation} 
\Phi(\eta,x)=A+W(\eta,x)\eta^{-\nu-\frac12}+B\eta^{-2\nu}+O(\eta^{-\nu-\frac32})
\end{equation}
This asymptotic expansion may be differentiated term by term in space as
often as desired. 
\end{theorem}
\noindent
Note that the third explicit term in this asymptotic expansion is 
often no larger than the error term. The function $W$ can be prescribed 
freely.

\section{Late-time asymptotics for a general equation of state}
\label{asymplatetimegen} 

It will now be investigated how the results of the previous section can be 
extended to the case of a more general equation of state. The class of 
equations of state which will be treated is defined by requiring that they 
admit an asymptotic expansion of the form
\begin{equation}\label{eqnofstate2}
f(\epsilon)\sim w\epsilon+\sum_{j=1}^\infty f_j\epsilon^{a_j}
\end{equation} 
for $\epsilon\to 0$. Here $w\ge 0$, the coefficients $a_j$ are all greater 
than one and form an increasing sequence. To ensure the positivity of $f'$
it is assumed that if $w=0$ the coefficient $f_1$ is positive. This form 
of the equation of state may be compared with that of (\ref{eqnofstate}).
It is further assumed that this expansion retains its validity when 
differentiated term by term as often as desired. An example is given by the 
polytropic equation of state (\ref{polytropic}). In that case $w=0$, $f_1=K$ 
and $a_1=\frac{n+1}{n}$. With this assumption information can be obtained on 
the leading order asymptotics of the background solution as $\eta\to\infty$. 
To simplify the notation define $\sigma=a_1-1$. It is convenient to use the
mass density once more, writing (\ref{massdensity}) in the equivalent form
\begin{equation}\label{massdensity2}
m(\epsilon)=\exp\left\{-\int_\epsilon^1 (\xi+f(\xi))^{-1}d\xi \right\}
\end{equation}
Then $m(\epsilon)$ has an expansion about $\epsilon=0$ where the leading term
is proportional to $\epsilon^{\frac{1}{w+1}}$. In particular, when $w=0$ the 
leading term is linear. Using the fact that $m$ is proportional to $a^{-3}$ 
for any equation of state leads to an asymptotic expansion for $\epsilon$ in 
terms of $a$. Putting this information into (\ref{SecondFriedmann}) shows that
$a(\eta)$ has an expansion in the limit $\eta\to\infty$ with the leading 
term proportional to $\eta^{\frac{2}{3w+1}}$. Finally it follows that 
$\epsilon$ and ${\cal H}$ have expansions with leading terms proportional
to $\eta^{-\frac{6(1+w)}{1+3w}}$ and $\eta^{-1}$ respectively. With the leading 
asymptotics of the background solution having been determined it is possible 
to derive asymptotics for the coefficients in the equation for $\Phi$. 

As in the case of a linear equation of state it is convenient to treat 
homogeneous and zero-mean solutions separately. The homogeneous solutions
will be analysed first. This leads to consideration of the equation obtained 
from (\ref{MainEquation}) by omitting the term containing spatial derivatives. 
It is convenient here to exclude the case of a linear equation of state
which was previously analysed so as to ensure that $\sigma$ is defined 
uniquely in terms of the equation of state. The coefficients satisfy:
\begin{equation}
3(1+f'(\epsilon))=3(1+w+(\sigma+1)f_1\eta^{-\beta})+o(\eta^{-\beta})
\end{equation} 
and 
\begin{equation}
3\left(f'(\epsilon)-\frac{f(\epsilon)}{\epsilon}\right)
=3f_1\sigma\eta^{-\beta}+o(\eta^{-\beta})
\end{equation}
where $\beta=\frac{6\sigma(1-w)}{1+3w}$. Define
\begin{equation}
F=\frac12\Phi'^2+\alpha\eta^{-2-\beta}\Phi^2
\end{equation}
where $\alpha$ is a positive constant which needs to be chosen appropriately 
in what follows. Computing the derivative of $F$ with respect to 
$\eta$ and using the equation gives a sum of terms involving $\Phi'^2$, 
$\Phi^2$ and $\Phi\Phi'$. The aim is to show that $F$ is bounded and to do 
this it suffices to consider arbitrarily late times. The leading order terms 
in the coefficients of $\Phi^2$ and $\Phi'^2$ are 
$-\frac{6(1+w)}{1+3w}\eta^{-1}$ and 
$-A\eta^{-3-\beta}$ respectively, where $A$ is a positive constant. The 
coefficient of $\Phi\Phi'$ has a leading term proportional to $\eta^{-2-\beta}$ 
for a general choice of $\alpha$. However if $\alpha$ is chosen to be half the 
coefficient of the leading order term in the expansion of the coefficient of 
$\Phi$ in (\ref{MainEquation}) then a cancellation occurs and the coefficient 
becomes $o(\eta^{-2-\beta})$. This choice is made here. The aim is to show that 
the term containing $\Phi\Phi'$ can be absorbed by the sum of the other two so 
as to leave a non-positive remainder. To do this 
the inequality
\begin{equation}
|\eta^{-2-\beta}\Phi\Phi'|\le\frac12 (\eta^{-1}\Phi'^2
+\eta^{-3-2\beta}\Phi^2)
\end{equation}
is used. The powers of $\eta$ which arise from this inequality match those
in the leading order terms in the coefficients of the manifestly negative
terms in the expression for the derivative of $F$ with respect to $\eta$. 
Thus at late times the cross-term can be absorbed in the terms with the 
desired sign. The conclusion is that $F$ is bounded. In fact this can be 
improved somewhat. The derivative of $F$ can be estimated above by 
$-2\gamma \eta^{-1}F$ for any positive constant $\gamma<2\nu+1$. 
This means that $\Phi'$ decays like $\eta^{-\gamma}$. It can be concluded that 
$\Phi$ is bounded. From the evolution equation for $\Phi$ and the boundedness 
statements already
obtained it follows that $(\eta^{2\nu+1}\Phi')'$ is integrable. Thus 
$\Phi=A+B\eta^{-2\nu}+\ldots$ for constants $A$ and $B$ and the leading order 
behaviour is as in the case of a linear equation of state. 

It turns out to be useful for the analysis of the zero-mean solutions in the
expanding direction to introduce a new time variable $\tau$ satisfying the 
relation $d\tau/d\eta=\sqrt{f'(\epsilon)}$. Substituting the asymptotics of 
$f'(\epsilon)$ in terms of $\eta$ into this provides an asymptotic 
expansion for $\tau$ in terms of $\eta$. For $w>0$ a linear relation is 
obtained in leading order while for $w=0$ and $\sigma\ne\frac13$ the 
expansion reads 
\begin{equation}\label{taudef}
\tau=C_1\eta^{1-3\sigma}+\tau_\infty+\ldots
\end{equation} 
for constants $C_1$ and $\tau_\infty$. Note that the second term in this
expansion is only smaller than the first for $\sigma<\frac13$. For $w=0$ and 
$\sigma=\frac13$ the power in this expression gets replaced by $\log\eta$. 
From these facts it can be seen that $\tau\to\infty$ for $\eta\to\infty$ when 
$w>0$ or when $w=0$ and $\sigma\le\frac13$. In contrast $\tau$ tends to the
finite limit $\tau_\infty$ for $\eta\to\infty$ when $w=0$ and 
$\sigma>\frac13$. This is a symptom of a bifurcation where the asymptotics of 
the linearized solution undergoes a major change. For convenience we say that 
the dynamics for an equation of state with an asymptotic expansion of the form 
(\ref{eqnofstate2}) is underdamped if $w>0$ or $\sigma<\frac13$, critical if 
$w=0$ and $\sigma=\frac13$ and overdamped if $w=0$ and $\sigma>\frac13$.

Next the late-time behaviour will be analysed for zero-mean solutions with
an equation of state corresponding to underdamped dynamics. The first step is 
to introduce the time variable $\tau$ into 
(\ref{MainEquation}) with the result:
\begin{equation}\label{MainEquationtau}
\Phi_{\tau\tau}+3Z\tilde{\cal H}\Phi_\tau
+3Y\tilde{\cal H}^2\Phi-\Delta\Phi=0
\end{equation}
where 
\begin{eqnarray}
&&Y=f'(\epsilon)-\frac{f(\epsilon)}{\epsilon},               \\
&&Z=1+f'(\epsilon)-\frac12\frac{(\epsilon
+f(\epsilon))f''(\epsilon)}{f'(\epsilon)}
\end{eqnarray}              
and
$\tilde{\cal H}=a^{-1}a_\tau$. Derivatives with respect to $\tau$ are denoted
by subscripts. Next the term containing $\Phi_\tau$ will
be eliminated by multiplying $\Phi$ by a suitable factor $\Omega^{-1}$. Choose
$\Omega$ to satisfy
\begin{equation}
\frac{\Omega_\tau}{\Omega}=-\frac32 Z\tilde{\cal H}
\end{equation}
For all three types the behaviour of $\Omega$ as a function of $a$
in the limit $\epsilon\to 0$ can be determined. The result is that the
leading order term in $\Omega$ is proportional to $a^{-\frac32(1+w)}$ for $w>0$ 
and proportional to $a^{-\frac32(1-\frac{\sigma}{2})}$ for $w=0$. The function 
$\Psi=\Omega^{-1}\Phi$ satisfies an equation of the form
\begin{equation}\label{taudynamics}
\Psi_{\tau\tau}=A(\epsilon)\tilde{\cal H}^2\Psi+\Delta\Psi
\end{equation}  
where $A(\epsilon)$ is 
a rational function of $\epsilon$, $f(\epsilon)$, $f''(\epsilon)$ and 
$f'''(\epsilon)$. Under the given assumptions on the equation of state it
is bounded. Proving this requires examining many terms but is routine. For 
example the only term containing the third derivative of $f$ is 
$\frac{3(\epsilon+f(\epsilon))^2f'''(\epsilon)}{2f'(\epsilon)}$. The leading 
order terms in the asymptotic expansions of numerator and denominator are
both proportional to $\epsilon^\sigma$. Note also that the leading order
term in the expansion for $\tilde{\cal H}$ is proportional to $\tau^{-1}$ for
any $\sigma<\frac13$. Note for comparison that $\tilde{\cal H}$ tends to a 
constant value as $\tau\to\infty$ in the case $\sigma=\frac13$.

Define an energy by   
\begin{equation}
E_5 (\tau)=\frac12\int\Psi_\tau^2+|\nabla\Psi|^2.
\end{equation}
Then using the same techniques as in previous energy estimates shows that 
there is a constant $C$ such that  
\begin{equation}
\frac{dE_5}{d\tau}\le C|A|\tilde{\cal H}^2E_5
\end{equation}
Using the information available concerning $A$ and $\tilde{\cal H}$ shows
that $E_5$ is bounded in the future. Taking derivatives of the equation and
using the same arguments as in previous cases shows that $\Psi$ and its 
derivatives of all orders with respect to $x$ and $\tau$ are bounded. It 
follows that any sequence of translates $\Psi(\tau+\tau_n)$ for a sequence
$\tau_n$ tending to infinity has a subsequence which converges on 
compact subsets to a limit $W$.  

Doing a Fourier transform of the equation (\ref{taudynamics}) in space leads
to the mode equation
\begin{equation}
\hat\Psi_{\tau\tau}=-|k|^2\hat\Psi+A\tilde{\cal H}^2\hat\Psi
\end{equation}
Introducing polar coordinates in the 
$(\frac{\hat\Psi}{|k|},\frac{\hat\Psi_\tau}{|k|})$-plane leads to the
system
\begin{eqnarray}
\frac{dr}{d\tau}&=&\frac{1}{|k|}Ar\tilde{\cal H}^{-2}
\sin\theta\cos\theta
\label{evolr2} \\ 
\frac{d\theta}{d\tau}&=&-|k|+\frac{1}{|k|}A
\tilde{\cal H}^{-2}\cos^2\theta
\label{evoltheta2}
\end{eqnarray}
This implies that
\begin{equation}
\theta (\tau)=-|k|(\tau-\bar\tau_k)+O(\tau^{-1})
\end{equation}
and 
\begin{equation}
r(\tau)=\bar W_k (1+O(\tau^{-1})) 
\end{equation}
for some constants $\bar\tau_k$ and $\bar W_k$. Arguing as in the case of a
linear equation of state leads to the relation $\Psi (\tau,x)=W(\tau,x)+o(1)$
where $W$ is a solution of the equation $W_{\tau\tau}=\Delta W$. Using the 
form of the leading order term in $\Omega$ as a function of $a$, it can be 
shown that the leading order term in $\Omega$ as a function of $\tau$ is
given by $\tau^{-\frac{3(1+w)}{1+3w}}=\tau^{-\nu-\frac12}$ and 
$\tau^{-\frac{3(1-\frac{\sigma}{2})}{(1-3\sigma)}}$ in the cases $w>0$ and
$w=0$ respectively. Note that the first of these reproduces the result in the 
case of a linear equation of state. It does not seem to be possible to write 
the expansion directly in terms of $\eta$ in such
a way that it gives more insight than the expression in terms of $\tau$.
The leading asymptotics of a zero-mean solution is obtained by taking a 
solution of the flat space wave equation, distorting the time variable by 
a diffeomorphism and multiplying by a power of the time coordinate which has
been explicitly computed.

Consider next the case $w=0$, $\sigma>\frac13$ (overdamped case). 
The time coordinate $\tau$ tends to the finite limit $\tau_\infty$ as 
$\eta\to\infty$.
Define $G=\tilde{\cal H}^{-2}\partial_\tau\tilde{\cal H}$. The function $G$ 
tends to the limit $\frac32(\sigma-\frac13)$ as $\tau\to\tau_\infty$. Let
\begin{equation}
E_6=\int \Phi_\tau^2+|\nabla\Phi|^2+\Lambda^2\tilde{\cal H}^2\Phi^2
\end{equation}
for a constant $\Lambda$ which remains to be chosen. For a constant $\lambda$
computing $\partial_\tau (a^{2\lambda}E_6)$ gives rise to a sum of expressions
containing $\Phi^2$, $\Phi_\tau^2$, $\Phi\Phi_\tau$ and $|\nabla\Phi|^2$.
Using the inequality 
\begin{equation}
\tilde{\cal H}^2\Phi_\tau\Phi\le \frac{1}{2\delta}\tilde{\cal H}\Phi_\tau^2
+\frac{\delta}{2}\tilde{\cal H}^3\Phi^2 
\end{equation}
leads to an inequality where the term involving $\Phi\Phi_\tau$ has been 
eliminated. To obtain some control on the energy by means of the inequality 
the coefficients $\Lambda$ and $\lambda$ should be chosen in such a way that 
all terms on the right hand side are manifestly non-positive. The conditions
for this to happen are the inequalities $\lambda\le 0$,
\begin{equation}\label{ineq4}
\frac{1}{2\delta}|\Lambda^2-3Y|\le 3Z-\lambda 
\end{equation}
and 
\begin{equation}\label{ineq5}
\frac{\delta}{2}|\Lambda^2-3Y|\le -\Lambda^2(\lambda+G)
\end{equation}
Note that these inequalities imply in particular that $\lambda<0$. 
Consider now the limit $\tau\to\tau_\infty$ where $Y$ behaves asymptotically 
like $f_1\sigma\epsilon^\sigma$ and $Z\to 1-\frac12\sigma$. In this 
limit the inequality (\ref{ineq5}) reduces to 
$\lambda\le -\frac32 (\sigma-\frac13)-\frac{\delta}{2}$. Suppose therefore
that $\lambda<-\frac32 (\sigma-\frac13)$. Then by choosing
$\delta$ sufficiently small it can be arranged that the limiting inequality
is satisfied. In the limit the inequality (\ref{ineq4}) reduces to
$\frac{\Lambda^2}{2\delta}\le 3-\frac32\sigma-\lambda$. Choose 
$\Lambda$ so that this inequality is satisfied strictly. With these choices 
both inequalities are satisfied strictly in the limit. For $\tau$ sufficiently
close to $\tau_\infty$ the coefficients in (\ref{ineq4}) and (\ref{ineq5})
are as close as desired to their limiting values. Making them close enough
ensures that these two inequalities continue to be satisfied. It follows 
that with these choices of the parameters $\partial_\tau (a^{2\lambda}E_6)$
is non-positive at late times. It can be concluded that $E_6=O(a^{-2\lambda})$.
This gives a limit on the growth rate of $E_6$ in terms of that of the scale
factor. As in previous cases corresponding estimates can be obtained for
derivatives and as a consequence pointwise estimates derived. It follows
that $\Phi=O(a^{-\lambda}\tilde{\cal H}^{-1})$. From what is known about the
background solution it follows that $\tilde{\cal H}$ is proportional to
$a^{\frac32 (\sigma-\frac13)}$. Thus if $\rho=-\frac32 (\sigma-\frac13)-\lambda$
then $\Phi=O(a^\rho)$. This power is positive but may be made as small as 
desired by choosing $\lambda$ suitably. By the usual methods similar bounds 
can be obtained for spatial derivatives of $\Phi$. 

To get more information about the asymptotics as $\tau\to\tau_\infty$ it is
convenient to rewrite the equation in terms of the new time variable 
$s=\tau_\infty-\tau$. The resulting equation is
\begin{equation}\label{MainEquations}
\Phi_{ss}-3Z\tilde{\cal H}\Phi_s
+3Y\tilde{\cal H}^2\Phi-\Delta\Phi=0
\end{equation} 
As $s\to 0$ the coefficient $Z$ tends to $1-\frac{\sigma}2$ while 
$\tilde {\cal H}$ and $Y\tilde {\cal H}^2$ are proportional in leading order 
to $s^{-1}$ and $s^{\frac{2}{3(\sigma-1/3)}}$ respectively. The last exponent is 
positive for any $\sigma>1/3$ so that the corresponding coefficient tends to 
zero as $s\to 0$. Let $B$ be a positive solution of the equation 
$\frac{dB}{ds}=-3Z\tilde {\cal H} B$. Then (\ref{MainEquations}) implies
the following integral equation:
\begin{equation}
\Phi_s=\frac{1}{B}\left(\bar\Phi_1+\int_0^s B(-3Y{\cal H}^2\Phi
+\Delta\Phi)\right)
\end{equation}
for a function $\bar\Phi_1(x)$. Here the fact has been used that the integral 
occurring in this equation converges. This follows from the fact that in 
leading order $B$ is proportional to $s^{-\frac{\sigma-2}{\sigma-1/3}}$ and the 
bounds already obtained for $\Phi$ and its derivatives. When $B^{-1}$ diverges 
faster than $s^{-1}$ in the limit $s\to 0$, which happens for 
$\sigma<\frac76$, the known bounds on $\Phi$ imply that $\bar\Phi_1=0$. Hence 
$\Phi$ is bounded in the limit $s\to 0$ in that case. When $\sigma>\frac76$ it 
can also be concluded that $\Phi$ is bounded. For $\sigma=\frac76$ a 
logarithmic divergence of $\Phi$ is not ruled out. In all cases the integral 
equation can be used to obtain an asymptotic expansion for $\Phi$. 
Schematically this expansion is of the form
\begin{equation}
\Phi (\eta,x)=\sum_i \Phi_i(x)\zeta_i (\eta).
\end{equation}
for some functions $\zeta_i$ with $\zeta_{i+1}(\eta)=O(\zeta_i(\eta))$ for 
each $i$. This is very different from the expansion in the limit 
$\eta\to\infty$ obtained when $w>0$ or $\sigma<\frac13$. In the present case, 
scaling the solution by a suitable function of $\eta$ gives a result which 
converges to a function of $x$ as $\eta\to\infty$. In the other case a
similar rescaling can lead to a profile which moves around the torus with 
constant velocity. (In general it leads to a superposition of profiles of 
this kind.) In the latter case there are waves which continue to propagate at 
arbitrarily late times. In the case $\sigma>\frac13$ the waves \lq freeze\rq. 
This is reminiscent of the late-time asymptotics of the gravitational field in 
spacetimes with positive cosmological constant (cf. \cite{rendall04}).  

To make this argument more concrete consider the special case where the
equation of state is $f(\epsilon)=f_1\epsilon^{\sigma+1}$ for some $\sigma$ 
between $\frac13$ and $\frac76$. Using the convergence of the integral it 
follows immediately that $\Phi(s,x)=\Phi_0(x)+O(s^2)$ for some function 
$\Phi_0$. Putting this information back into the integral equation gives
$\Phi(s,x)=\Phi_0 (x)+\frac{\sigma-1/3}{4-2\sigma}\Delta\Phi_0 (x)s^2
+\dots$. 

Consider finally the case $w=0$, $\sigma=\frac13$ (critical case). Then 
$\eta=\eta_0e^{\frac{\tau}{C_1}}+\ldots$ where $\eta_0$ is a constant and $C_1$ 
corresponds to the constant appearing in (\ref{taudef}). The arguments leading
to the estimate $\Phi=O(\tau^\rho)$ can be carried out as in the case
$\sigma>\frac13$. The only difference is that the limit $\tau\to\tau_\infty$
is replaced by $\tau\to\infty$. In the case $\sigma=\frac13$ the quantity
$\tilde{\cal H}$ tends to a constant for $\tau\to\infty$ and 
$Y\tilde{\cal H}^2$ is proportional to $e^{-\frac{6}{C_1}\tau}$ in leading order.
A quantity $B$ can be introduced as before and an integral equation obtained.
In this case $B$ is a decaying exponential. Unfortunately it does not seem to
be possible to use this integral equation to refine the asymptotics in this 
case and this matter will not be pursued further here. 

\vskip 10pt
\noindent
{\it Acknowledgements}

\noindent
The authors gratefully acknowledge the hospitality and financial support
of the Mittag-Leffler Institute where part of this work was carried out.

\end{document}